\documentclass[11pt,reqno]{article}
\usepackage{amssymb, amsmath, amsthm, amsfonts, amscd, epsfig, subfig}
\usepackage[dvipsnames]{xcolor}
\usepackage[utf8]{inputenc}
\usepackage[english]{babel}
\usepackage{bm}
\usepackage{bbm}
\usepackage{graphicx, epsfig}
\usepackage{geometry,graphicx,pgfplots}
\usepackage{hyperref}

\usepackage[export]{adjustbox}
\usepackage[titletoc,toc,title]{appendix}
\usepackage{algorithm}
\usepackage{mathtools}

\usepackage{afterpage}
\usepackage{comment}
\usepackage{cases}
\usepackage{indentfirst}

\newtheorem{theorem}{Theorem}[section]

\newtheorem{lemma}[theorem]{Lemma}                                                                                                                                                                                                                                                                             
\newtheorem{prop}[theorem]{Proposition}

\newtheorem{notation}{Notation}[section]

\def\ep{\varepsilon}

\newcommand{\Om}{\Omega}

\newcommand{\Bx}{\mathbf{x}}

\newcommand{\Bn}{\mathbf{n}}

\newcommand{\RR}{\mathbb{R}}

\newcommand{\NN}{\mathbb{N}}

\newcommand{\p}{\partial}

\newcommand{\pd}[2]{\frac {\p #1}{\p #2}}
\newcommand{\ds}{\displaystyle}
\newcommand{\eqnref}[1]{(\ref {#1})}

\renewcommand{\qed}{\hfill $\Box$ \medskip}

\newcommand{\beq}{\begin{equation}}
\newcommand{\eeq}{\end{equation}}
\newcommand{\RN}[1]{%
  \textup{\uppercase\expandafter{\romannumeral#1}}%
}

\newcommand{\eps}{{\varepsilon}}

\newcommand{\RomanNumeralCaps}[1]
    {\MakeUppercase{\romannumeral #1}}

\numberwithin{equation}{section}
\numberwithin{figure}{section}

\begin{document}
\title{On the first Steklov--Dirichlet eigenvalue on eccentric annuli in general dimensions
 \thanks{\footnotesize
JH and ML are grateful for support from the National Research Foundation of Korea (NRF) grant funded by the Korean government (MSIT) (NRF-2021R1A2C1011804).
DS is supported by the National Research Foundation of Korea (NRF) grant funded by the Korea Government (MSIT) (No.
2021R1C1C2005144).}}

\date{}

\author{
Jiho Hong\thanks{Department of Mathematical Sciences, Korea Advanced Institute of Science and Technology, 291 Daehak-ro, Yuseong-gu, Daejeon 34141, Republic of Korea (jihohong@kaist.ac.kr, mklim@kaist.ac.kr).}\and
Mikyoung Lim\footnotemark[2]\and
Dong-Hwi Seo\thanks{Research Institute of Mathematics, Seoul National University, 1 Gwanak-ro, Gwanak-gu, Seoul 08826, Republic of Korea (donghwi.seo26@gmail.com).}
}

\maketitle


\begin{abstract}
\noindent
We consider the Steklov--Dirichlet eigenvalue problem on eccentric annuli in Euclidean space of general dimensions. In recent work by the same authors of this paper \cite{HLS:2022:FSD}, a limiting behavior of the first eigenvalue, as the distance between the two boundary circles of an annulus approaches zero, was obtained in two dimensions. We extend this limiting behavior to general dimensions by employing bispherical coordinates and expressing the first eigenfunction as a Fourier--Gegenbauer series.

\end{abstract}

\noindent {\footnotesize {\emph{2020 Mathematics Subject Classification}.}  35P15, 49R05, 42C10.}

\noindent {\footnotesize {\bf Key words.} 
Steklov--Dirichlet eigenvalue, Eccentric annulus, Eigenvalue estimate, Bispherical coordinates, Fourier--Gegenbauer series}

%
%

\section{Introduction}\label{sec1}

Let $\Omega\subset \mathbb{R}^d$ be a smooth bounded domain with two boundary components, $C_1$ and $C_2$. We consider an eigenvalue problem for the Laplacian operator with a mixed boundary condition:
\begin{align} 
\begin{cases} \label{eqn:Steklov-Dirichlet}
    \ds \Delta u = 0   &\ds\text{in   } \Omega,\\
  \ds   u = 0 & \ds\text{on   }  C_1,\\
  \ds   \frac{\partial u}{\partial \Bn} = \sigma u &\ds\text{on   } C_2,
\end{cases}
\end{align}
where $\Bn$ denotes the unit outward normal vector to $\p\Om$. If (\ref{eqn:Steklov-Dirichlet}) with a real constant $\sigma$ admits a non-trivial solution, we call $\sigma$ a Steklov--Dirichlet eigenvalue and $u$ the associated eigenfunction. There are only discrete eigenvalues $ 0<\sigma_1(\Omega)\le \sigma_2(\Omega) \le \cdots \rightarrow \infty,$ provided that $C_1\neq \emptyset$ (see, for example, \cite{Agranovich:2006:MPS}). 
When $C_1 = \emptyset$, the eigenvalue problem \eqnref{eqn:Steklov-Dirichlet} become the classical Steklov eigenvalue problem \cite{stekloff:1902:FPM}; we refer to the survey articles \cite{Girouard:2017:SGS,CGGS:2022:SRD} for details and more references on this topic. We are interested in the first Steklov--Dirichlet eigenvalue $\sigma_1(\Omega)$, which has the variational characterization as follows (see, for example, \cite{Bandle:1980:IIA}):
\begin{align} \label{variational characterization}
    \sigma_1(\Omega) = \inf \left\{ {\ds\int_{\Omega}\left|\nabla v\right|^2\, {\rm d}x} \,\Big|\, v \in H^1(\Omega),\ v=0 \textnormal{ on } C_1,\ \mbox{and }{\ds\int_{C_2}v^2 \,{\rm d}S=1}  \right\}.
\end{align}

The Steklov--Dirichlet eigenvalue problem admits various physical interpretations. For instance, Steklov--Dirichlet eigenfunctions can represent vibration modes of a partially free membrane fixed along $C_1$ with all its mass concentrated along $C_2$ \cite{Hersch:1968:EPI}. Also, the problem models the stationary heat distribution in $\Omega$ under the conditions that the temperature along $C_1$ is kept to zero and that the heat flux through $C_2$ is proportional to the temperature \cite{Banuelos:2010:EIM, KKKNPPS:2014:LVA}.

Many authors have been concerned with the geometric dependence of the Steklov--Dirichlet eigenvalues, which is the focus of this paper. In 1968, Hersch and Payne considered a Steklov--Dirichlet eigenvalue problem on bounded doubly connected domains in $\mathbb{R}^2$ and derived upper and lower bounds on the first eigenvalue \cite{Hersch:1968:EPI}. Dittmar obtained a formula for the reciprocal sum of eigenvalues for planar domains and induced an isoperimetric inequality \cite{Dittmar:1998:IIS}. Using conformal mapping theory, Dittmar and Solynin showed a lower bound of the first eigenvalue for a class of bounded doubly connected domains in $\mathbb{R}^2$ \cite{Dittmar:2003:MSE} (see also \cite{Dittmar:2005:EPC}). Ba\~{n}uelos et al. \cite{Banuelos:2010:EIM} compared the Steklov--Dirichlet and Steklov--Neumann eigenvalues for a class of domains in $\mathbb{R}^d$; this result is reminiscent of classical inequalities between the Dirichlet and Neumann eigenvalues.  Recently, there have been isoperimetric results \cite{VS:2020:EPL, Seo:2021:SOP, Ftouhi:2022:WPS, GPP:2022:IIF}, which will be discussed more later, spectral stability results \cite{PPS:2021:SRS, Mic:2022:SDS}, and estimates of the Riesz means of mixed Steklov eigenvalues \cite{HL:2020:EBM}.

In the present paper, we consider the first Steklov--Dirichlet eigenvalue on an eccentric annulus $\Om$ in $\RR^m$ ($m\geq 2$), where the inner and outer boundaries of $\Om$ are the spheres with radius $r_1$ and $r_2$, respectively. The two radii $0<r_1<r_2$ are fixed. We denote by $t$ the distance between the centers of $B_1^t$ and $B_2$.
 More precisely, we set $\Om=B_2\setminus\overline{B_1^t}$ with
$$B_1^t=B(t\mathbf{e}_1,r_1),\quad B_2=B(\mathbf{0},r_2)\quad\mbox{for }0\leq t<r_2-r_1,$$
where $B(\mathbf{x},r)$ means the ball centered at $\mathbf{x}$ with radius $r$ and $\mathbf{e}_1$ is the unit vector $(1,0,\dots,0)$. 
 Note that $\overline{B_1^t}\subset B_2$. For illustrations of $B_1, B_1^t$ and $B_2$, see Fig. \ref{fig:geometry}. By imposing the Dirichlet condition on $\p B_1^t$ and the Steklov condition on $\p B_2$ in \eqnref{eqn:Steklov-Dirichlet}, the first Steklov--Dirichlet eigenvalue and associated eigenfunction, $\sigma_1^t$ and $u_1^t$, respectively, satisfy
\begin{align} \label{problem}
\begin{cases}
    \ds \Delta u_1^t = 0   &\ds\text{in   } B_2\setminus \overline{B_1^t},\\
  \ds   u_1^t = 0 & \ds\text{on   } \partial B_1^t,\\
  \ds   \frac{\partial u_1^t}{\partial \Bn} = \sigma_1^t u_1^t &\ds\text{on   } \partial B_2.
\end{cases}
\end{align}
Differentiability for $\sigma_1^t$ and $u_1^t$ in $t\in[0, r_2-r_1)$ and the shape derivative of $\sigma_1^t$ were obtained in \cite{HLS:2022:FSD}.

\begin{figure}[ht]
    \centering
    \begin{tikzpicture}

\coordinate  (X) at (0,0);
\coordinate (D) at (1.732,1);

\fill[gray!40,even odd rule] (X) circle (0.7) (X) circle (2);	

\draw (X) circle(2cm);
\draw (X) circle (0.7cm);
\draw (X) --(D);
\draw (X) --(-0.7,0);

\node[above, scale=0.8] at (0.866,0.5) {$r_2$};
\node[above, scale=0.8] at (-0.35,0) {$r_1$};
\node at (0,2.25){$B_2$};
\node at (0,0.95){$B_1$}; 
\node[below] at (0,0){$O$};
\foreach \point in {X}
	\fill [black] (\point) circle (1.5pt);

	\end{tikzpicture}
	\hskip 1.5cm
        \begin{tikzpicture}
\coordinate  (C) at (0,0);
\coordinate (D) at (2,0);
\coordinate (X) at (1,0);
\fill[gray!40,even odd rule] (X) circle (0.7) (C) circle (2);
\draw (C) circle(2cm);
\draw (X) circle (0.7cm);

\draw (X) --(C);
\node at (0.5,0.2) {$t$};

\node at (1,0.95){$B_1^t$};
\node at (0,2.25){$B_2$};
\node[below] at (0,0){$O$};

\foreach \point in {X,C}
	\fill [black] (\point) circle (1.5pt);

	\end{tikzpicture}
    \caption{A concentric annulus $B_2\setminus B_1$ (left) and an eccentric annulus $\Om=B_2\setminus\overline{B_1^t}$ (right).  The parameter $t$ means the distance between the centers of $B_1^t$ and $B_2$ and, thus, belongs to $[0, r_2-r_1)$.}  \label{fig:geometry}
\end{figure}
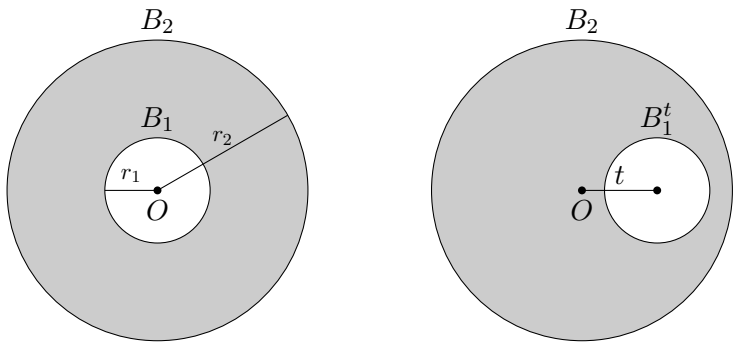

For the Steklov--Dirichlet eigenvalue problem \eqnref{problem}, Santhanam and Verma showed that the first eigenvalue $\sigma_1^t$ for $\Omega \subset \mathbb{R}^d, d\ge 3$ attains the maximum at $t=0$, that is, when $\Omega$ is the concentric annulus \cite{VS:2020:EPL}. Then, Seo and Ftouhi verified independently that the result of  Santhanman and Verma holds for $\mathbb{R}^2$ \cite{Seo:2021:SOP, Ftouhi:2022:WPS}. Furthermore, Seo  \cite{Seo:2021:SOP} generalized this isoperimetric result to two-point homogeneous space $M$, given that $r_2$ is less than the half of the injectivity radius of $M$, and Gavitone et al.  \cite{GPP:2022:IIF} to more general domains in Euclidean space. 

In \cite{HLS:2022:FSD}, the same authors of this paper obtained a lower bound for the limit inferior of $\sigma_1^t$ in $\RR^2$ as the distance between the boundary circles of the annulus approaches zero:
\beq\label{lower:liminf}
\liminf_{t\to (r_2-r_1)^-}\sigma_1^t\ge\frac{\ds r_1}{\ds2r_2(r_2-r_1)}\quad\mbox{for }\Om=B_2\setminus B_1^t\subset\RR^2.
\eeq
See also \cite{Dittmar:2003:MSE}.

The aim of this paper is to extend \eqnref{lower:liminf} to general dimensions.
There is a fundamental difficulty in deriving a lower bound for the first Steklov--Dirichlet eigenvalue as $\sigma_1^t$ is given by an infimum, as can be seen from the characterization \eqnref{variational characterization}.
We overcome the difficulty by employing bispherical coordinates in general dimensions and expressing the first eigenfunction $u_1^t$ as a Fourier--Gegenbauer series. By a careful asymptotic treatment of the series expression for $\eps=r_2-r_1-t\ll 1$, we derive the following. The proof will be provided in section \ref{sec:main proof}.
\begin{theorem}\label{thm:main}
For the Steklov--Dirichlet eigenvalue problem (\ref{problem}) on $\Om=B_2\setminus\overline{B_1^t}\subset \mathbb{R}^{n+2}, n\ge 1$, the first eigenvalue $\sigma_1^t$ satisfies 
\beq\label{ineq:sdeig:lowerbound}
\liminf_{t\to (r_2-r_1)^-}\sigma_1^t\ge\frac{(n+1)r_1-nr_2}{2r_2(r_2-r_1)}.\eeq
\end{theorem}

Note that various eigenvalue problems with Dirichlet boundary conditions have been extensively studied on eccentric annuli. These problems include the Dirichlet Laplacian problem \cite{RS:1998:IME, HKK:2001:POO, AA:2005:TFL, AV:2013:TFL, ABS:2018:SMF, AR:2019:FDE}, the Dirichlet $p$-Laplacian problem \cite{CM:2015:EOP, ABS:2018:SMF}, the Dirichlet fractional Laplacian problem \cite{DFW:2021:FHF}, and the Zaremba problem \cite{AAK:2021:SVR}. For these eigenvalue problems on an eccentric annulus, the first eigenvalue monotonically decreases as the distance between the two boundary spheres increases. This behavior is similar to that observed for the Dirichlet heat trace \cite{EH:2016:POO} and the Dirichlet heat content \cite{Li:2022:POO}.

The remainder of this paper is organized as follows. In section 2, we introduce the bispherical coordinates in general dimensions and the Gegenbauer polynomials. Section 3 is devoted to deriving a Fourier--Gegenbauer series expansion for the first Steklov--Dirichlet eigenfunction by using the bispherical coordinates. In section 4, we investigate the asymptotic behavior for the expansion coefficients of the first eigenfunction and prove the main theorem.

\section{Preliminaries}
\subsection{Bispherical coordinates} \label{sec: the first: bispherical coordinates}
Let $\alpha>0$. Any point $\mathbf{x}=(x_1,x_2)\in\RR^2$ in the Cartesian coordinates admits the bipolar coordinates $(\xi,\theta)\in\RR\times[0,2\pi)$ via the relation
\begin{align*}
x_1&=\frac{\alpha\sinh\xi}{\cosh\xi-\cos\theta}=:\operatorname{Bi}_{2,1}(\xi,\theta), \\
x_2&=\frac{\alpha\sin\theta}{\cosh\xi-\cos\theta}=:\operatorname{Bi}_{2,2}(\xi,\theta),
\end{align*}
where the poles are located at $\alpha\mathbf{e}_1$. We also write $\Bx=\Bx(\xi,\theta)$ to indicate the dependence of $\Bx$ on $(\xi,\theta)$. 
Similarly, the bispherical coordinates for $\mathbf{x}=(x_1,x_2,x_3)\in\RR^3$ are defined by 
\begin{align*}
x_1&=\frac{\alpha\sinh\xi}{\cosh\xi-\cos\theta}=:\operatorname{Bi}_{3,1}(\xi,\theta,\varphi_1),\\
x_2&=\frac{\alpha\sin\theta\cos\varphi_1}{\cosh\xi-\cos\theta}=:\operatorname{Bi}_{3,2}(\xi,\theta,\varphi_1),\\
x_3&=\frac{\alpha\sin\theta\sin\varphi_1}{\cosh\xi-\cos\theta}=:\operatorname{Bi}_{3,3}(\xi,\theta,\varphi_1).
\end{align*}

We generalize the bispherical coordinates to $\RR^{n+2}$, $n\geq1$, by the mapping
$\mathbf{x}=\mathbf{x}(\xi,\theta,\varphi_1,\dots \varphi_{n}): \RR\times[0,\pi]^{n}\times[0,2\pi)\rightarrow \RR^{n+2}$ whose components are given by
\beq \label{bihyperspherical coordinates}
x_j=\operatorname{Bi}_{n+2,j}\left(\xi,\theta,\varphi_1,\dots \varphi_{n}\right)\quad\mbox{for each } j=1,\dots,n+2,
\eeq
where the functions $\operatorname{Bi}_{n+2,j}$ are recursively defined by
\beq\notag
\begin{aligned}
\operatorname{Bi}_{n+2,j}\left(\xi,\theta,\varphi_1,\dots \varphi_{n}\right)
&=\operatorname{Bi}_{n+1,j}\left(\xi,\theta,\varphi_1,\dots \varphi_{n-1}\right)\qquad \mbox{for }j=1,\dots,n,\\
\operatorname{Bi}_{n+2,n+1}\left(\xi,\theta,\varphi_1,\dots \varphi_{n}\right)
&=\operatorname{Bi}_{n+1,n+1}\left(\xi,\theta,\varphi_1,\dots \varphi_{n-1}\right) \cos\varphi_n,\\
\operatorname{Bi}_{n+2,n+2}\left(\xi,\theta,\varphi_1,\dots \varphi_{n}\right)
&=\operatorname{Bi}_{n+1,n+1}\left(\xi,\theta,\varphi_1,\dots \varphi_{n-1}\right)\sin\varphi_n.
\end{aligned}
\eeq
For instance, it holds for $n=2$ that
\begin{align*}
\operatorname{Bi}_{4,1}(\xi,\theta,\varphi_1,\varphi_2)&=\frac{\alpha\sinh\xi}{\cosh\xi-\cos\theta},\quad \operatorname{Bi}_{4,2}(\xi,\theta,\varphi_1,\varphi_2)=\frac{\alpha\sin\theta\cos\varphi_1}{\cosh\xi-\cos\theta},\\
\operatorname{Bi}_{4,3}(\xi,\theta,\varphi_1,\varphi_2)&=\frac{\alpha\sin\theta\sin\varphi_1\cos\varphi_2}{\cosh\xi-\cos\theta},
\quad \operatorname{Bi}_{4,4}(\xi,\theta,\varphi_1,\varphi_2)=\frac{\alpha\sin\theta\sin\varphi_1\sin\varphi_2}{\cosh\xi-\cos\theta}.
\end{align*}
See Fig. \ref{fig:coord:3d} for level surfaces of the bispherical coordinates.

\begin{figure}[ht!]
\begin{center}

\begin{tikzpicture}[scale=0.43]

\coordinate  (C) at (3.537742925, 0);
\coordinate (X) at (2.125, 0);
\fill[gray!40,even odd rule] (X) circle (1) (C) circle (3);

\node at (2,1.37){$B_1^t$};
\node at (3.537742925,3.35){$B_2$};

\draw[thick] (3.537742925, 0) circle (3);
\draw[thick] (2.125, 0) circle (1);
\draw[thick] (-3.537742925, 0) circle (3);
\draw[thick] (-2.125, 0) circle (1);

\draw (0, -6.0) -- (0, 6.0);
\draw (-9.27, 0) -- (9.27, 0);

\fill (0, 6.0) -- (-0.1, 5.8) -- (0.1, 5.8);
\fill (9.27, 0) -- (9.07, 0.1) -- (9.07, -0.1);
\draw (9.27, -0.5) node {$x_1$};
\draw (0.2, 6.3) node {$x_2,\dots, x_{n+2}$};
\draw (-0.3, -0.35) node {$O$};

\draw[dashed, domain=30:150] plot ({2.1650635*cos(\x)},{-1.08253175+2.1650635*sin(\x)});
\draw[dashed, domain=-150:-30] plot ({2.1650635*cos(\x)},{1.08253175+2.1650635*sin(\x)});
\draw[dashed, domain=-40:220] plot ({0+2.40117*cos(\x)}, {1.5+2.40117*sin(\x)});
\draw[dashed, domain=140:400] plot ({0+2.40117*cos(\x)}, {-1.5+2.40117*sin(\x)});
\draw [dashed, domain=-5:70] plot ({5.34*cos(\x)}, {5.34*sin(\x)-5});
\draw [dashed, domain=175:250] plot ({5.34*cos(\x)}, {5.34*sin(\x)+5});
\draw [dashed, domain=110:185] plot ({5.34*cos(\x)}, {5.34*sin(\x)-5});
\draw [dashed, domain=290:365] plot ({5.34*cos(\x)}, {5.34*sin(\x)+5});

\fill (1.875, 0) circle (0.07);
\fill (-1.875, 0) circle (0.07);
\fill (0, 0) circle (0.07);

\draw (5, 0.46) node {$\xi=\xi_1$};
\draw (8.1, 1.45) node {$\xi=\xi_2$};

\draw (3.1098, 0.173648) -- (3.7098, 0.4);
\draw (6.29006, 1.25) -- (6.89006, 1.45);

\fill (3.1098, 0.173648) -- (3.1098+0.156693501419659, 0.173648+0.159521618011000) -- (3.1098+0.221633395260595, 0.173648-0.0296418303291268);
\fill (6.29006, 1.25) -- (6.29006+0.156693501419659, 1.25+0.159521618011000) -- (6.29006+0.221633395260595, 1.25-0.0296418303291268);

\end{tikzpicture}

\end{center}
\vskip -4mm
\caption{\label{fig:coord:3d}$\xi$-level surfaces (thick curves) and $\theta$-level surfaces (dashed curves) of the bispherical coordinate system in $\mathbb{R}^{n+2}$.}
\end{figure}
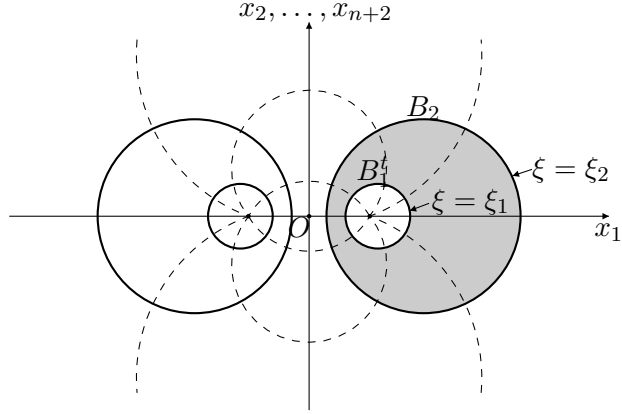

For a function $u$, the outward normal derivative on $\xi=\tilde{\xi}$ for a fixed $\tilde{\xi}>0$ satisfies
\beq\label{normal:u:3d}
\pd{u}{\mathbf{n}}=-\frac{1}{h({\xi},\theta)}\pd{u}{\xi}\Big|_{\xi=\tilde{\xi}}\quad\mbox{with } h(\xi,\theta) = \frac{\alpha}{\cosh\xi-\cos\theta}.
\eeq
Here, $h(\xi,\theta)$ is the scale factor of the bispherical coordinate system for the parameter $\xi$.

Now, we express $\Delta u$ in bispherical coordinates (see \eqnref{bihyperspherical coordinates}) for a given function $u\in C^{\infty}(\mathbb{R}^{n+2})$. For simplicity, we introduce the notation $\mathbf{y}=(y_1,y_2,y_3,\dots,y_{n+2})=(\xi,\theta,\varphi_1,\dots,\varphi_n)$. For $\mathbf{x}=(x_1,\dots,x_{n+2})$ in Cartesian coordinates, we define
$$g_{ij}:=\left\langle \pd{\mathbf{x}}{y_i},\, \pd{\mathbf{x}}{y_j}\right\rangle_{\RR^{n+2}}\quad \mbox{for }i,j=1,\dots,n+2.$$
It holds that $g_{ij}=0$ for $i\neq j$ and 
\begin{align*}
    g_{11}&=\frac{\alpha^2}{(\cosh\xi-\cos\theta)^2},\quad g_{22}=\frac{\alpha^2}{(\cosh\xi-\cos\theta)^2},\\
    g_{33}&=\frac{\alpha^2\sin^2\theta}{(\cosh\xi-\cos\theta)^2},\quad
g_{k+2,k+2}=\frac{\alpha^2\sin^2\theta\sin^2\varphi_1\cdots \sin^2\varphi_{k-1}}{(\cosh\xi-\cos\theta)^2}\quad \mbox{for }k=2,\dots,n,
\end{align*}
and, thus, $$\sqrt{|\mathbf{g}|}:=\sqrt{\det (g_{ij})}=\frac{\alpha^{n+2}\sin^{n}\theta\sin^{n-1}\varphi_1\cdots\sin^2\varphi_{n-2}\sin\varphi_{n-1}}{(\cosh\xi-\cos\theta)^{n+2}}.$$
Let $g^{ij}$ be the components of the inverse of the metric operator $(g_{ij})$, which is diagonal. We have $g^{jj}=g_{jj}^{-1}$. Then, the formula for the Laplace--Beltrami operator leads to
\begin{align}\notag
\Delta u
&=\frac{1}{\sqrt{|\mathbf{g}|}}\pd{}{y_i}\left(\sqrt{|\mathbf{g}|} g^{ij}\pd{u}{y_j}\right)\\\notag
 &=
    \frac{1}{\sqrt{|\mathbf{g}|}}\bigg[\frac{\partial}{\partial \xi}\left(\frac{\alpha^{n}\sin^{n}\theta \sin^{n-1}\varphi_1\cdots \sin\varphi_{n-1}}{(\cosh\xi-\cos\theta)^{n}}\frac{\partial u}{\partial \xi} \right)
 +\frac{\partial}{\partial \theta}\left(\frac{\alpha^{n}\sin^{n}\theta \sin^{n-1}\varphi_1\cdots \sin\varphi_{n-1}}{(\cosh\xi-\cos\theta)^{n}}\frac{\partial u}{\partial \theta} \right)\\ \label{laplacian}
   &\qquad\qquad +\sum_{i=1}^{n}\frac{\partial}{\partial \varphi_i}\left(\sqrt{|\mathbf{g}|}g^{i+2, i+2}\frac{\partial u}{\partial \varphi_i} \right)\bigg]. 
\end{align}
The bispherical coordinates allow $R$-separation of the Laplace equation so that a function $v$ of the form
$$v=(\cosh\xi-\cos\theta)^{\frac{n}{2}}\cdot\Xi(\xi)\Theta(\theta)\Psi_1(\varphi_1)\cdots\Psi_{n}(\varphi_{n})$$ permits the separation of the equation $\Delta v=0$ into $n+2$ number of ordinary differential equations, which can be derived from \eqnref{laplacian}; refer to \cite[Section \RomanNumeralCaps{4}]{Moon:1988:FTH} for the case $n=1$.

\subsection{Gegenbauer polynomials}
The Gegenbauer polynomials, also called ultraspherical polynomials, $G_m^{(\lambda)}(s)$, $s\in[-1, 1]$ with $m\in\NN \cup \{0\}$, $\lambda\in (0,\infty)$ are given by the generating relation (see, for instance, (4.7.23) in \cite{Sze:1975:OP})
$$(1-2st+t^2)^{-\lambda}=\sum_{m=0}^\infty G_{m}^{(\lambda)}(s)\,t^m$$
for $s\in (-1,1)$ and $t\in[-1, 1]$. In some literature, they are denoted by $P_n^{(\lambda)}$ (for instance, \cite{Sze:1975:OP, SW:1971:IFA}).  If $\lambda=\frac{1}{2}$, $\left\{G_m^{\left(1/2\right)}(s)\right\}_{m\ge 0}$ for $s\in [-1, 1]$ become the Legendre polynomials. More generally, $\left\{G_m^{\left(n/2\right)}(s)\right\}_{m\ge 0}$ are related to the zonal spherical harmonics in $\mathbb{R}^{n+2}$ \cite[Theorem 2.14]{SW:1971:IFA}.
We list some properties of the Gegenbauer polynomials.
\begin{itemize}
    \item We have the recurrence relation (see (4.7.17) in \cite{Sze:1975:OP}):
    \begin{equation}\label{Gegen:recur}
\begin{gathered}
        G_0^{(\lambda)}(s)=1, \quad G_1^{\lambda}(s)=2s\lambda,\\
    mG_m^{(\lambda)}(s)-2(m+\lambda-1)sG_{m-1}^{(\lambda)}(s)+(m+2\lambda-2)G_{m-2}^{(\lambda)}(s)=0 \text{ for all } m\ge 2.    
    \end{gathered}
    \end{equation}
     
    \item The derivatives are computed as what follows (see (4.7.14) in \cite{Sze:1975:OP}):
    \begin{align} \label{gegen:deri}
    \frac{d}{ds}G_m^{(\lambda)}(s)=2\lambda G_{m-1}^{(\lambda+1)}(s) \text{ for all } m\ge 1.
    \end{align}
    \item $G_{m}^{(\lambda)}(s)$ is a polynomial of degree $m$ and it is orthognoal to all polynomials of degree at most $m-1$ in the weighted space $L^2([-1,1];(1-s^2)^{\lambda-1/2}\,ds)$. In particular, $\left\{G_{m}^{(\lambda)}(s)\right\}_{m\ge 0}$ is complete and orthogonal in $L^2([-1,1];(1-s^2)^{\lambda-1/2}\,ds)$ \cite[Corollary IV 2.17]{SW:1971:IFA}.
    \item We have the following differential equation for $y(s)=G_m^{(\lambda)}(s)$ (see (4.7.5) in \cite{Sze:1975:OP}):
    \begin{align}\label{Gegen:ode}
        (1-s^2)y''-(2\lambda+1)sy'+m(m+2\lambda)y=0.
    \end{align}
    Equivalently, it holds that
     \begin{align}\label{Gegen:ode2}
        \left((1-s^2)^{\lambda+\frac{1}{2}}y'\right)'+m(m+2\lambda)(1-s^2)^{\lambda-\frac{1}{2}}y=0.
    \end{align}
    \item We have the following Rodrigues type formula (see (4.7.12) in \cite{Sze:1975:OP}):
    \begin{align}\label{expan:gegen}
     (1-s^2)^{\lambda-\frac{1}{2}}G_m^{(\lambda)}(s)=\frac{(-2)^m}{m!}\frac{\Gamma(m+\lambda)}{\Gamma(\lambda)}\frac{\Gamma(m+2\lambda)}{\Gamma(2m+2\lambda)}\frac{d^m}{ds^m}(1-s^2)^{m+\lambda-\frac{1}{2}}.   
    \end{align}
   \item The maximum modulus of $G_{m}^{(\lambda)}(s)$ happens at $s=\pm 1$ and
   \begin{align}\label{gegen_maximum}
   \left|G_m^{(\lambda)}(s) \right|\le \left|G_m^{(\lambda)}(1) \right|=\frac{\Gamma(m+2\lambda)}{\Gamma(m+1)\Gamma(2\lambda)}.
   \end{align}
   See (4.7.3) in \cite{Sze:1975:OP} and \cite[Theorem 7.33.1]{Sze:1975:OP}. 
   \end{itemize}
Note that (\ref{gegen_maximum}) implies
\begin{align} \label{maximum bound}
    \left|G_m^{(\lambda)}(s) \right|\le Cm^k
\end{align}
for some constants $C=C(\lambda)>0$ and $k=k(\lambda)>0$.
Using the Rodrigues type formula or (4.7.15) in \cite{Sze:1975:OP}, we have
\begin{align}
\begin{split} \label{Gegen:norm}
\|G_m^{(\lambda)}\|^2_{\lambda-\frac{1}{2}}&:=\int_{-1}^1 \left(G_m^{(\lambda)}(s)\right)^2 (1-s^2)^{\lambda -\frac{1}{2}}\,ds\\
&=2^{1-2\lambda}\pi\frac{(\Gamma(\lambda))^{-2}\Gamma(m+2\lambda)}{(m+\lambda)\Gamma(m+1)}
\end{split}\\
&\le Cm^{k}
\end{align}
for some constants $C=C(\lambda)>0$ and $k=k(\lambda)>0$.

\section{The first eigenfunction $u_1^t$ in bishperical coordinates} \label{sec:first eigenfunction}

We choose the parameter $\alpha$ in the bispherical coordinates as \beq\label{def:alpha}
\alpha=\frac{\sqrt{\left((r_2+r_1)^2-t^2\right)\left((r_2-r_1)^2-t^2\right)}}{2t};
\eeq 
then, after appropriately applying the rotation and translation, the annulus (again, denoted by $\Om$) becomes $\Om=B_2\setminus\overline{B_1^t}$ with 
\beq\label{eqn:B1B2:2}
B_1^t=t_0\mathbf{e}_1+B(-t\mathbf{e}_1,r_1),\quad B_2=t_0\mathbf{e}_1+B(0,r_2)\quad\mbox{for some }t_0>0
\eeq
where
$\p B_1^t$ and $\p B_2$ are the $\xi$-level curves of $\xi_1$ and $\xi_2$, respectively, with (see Fig. \ref{fig:coord:3d})
\beq\label{def:xi_j}
\xi_j=\ln\bigg(\frac{\alpha}{r_j}+\sqrt{\left(\frac{\alpha}{r_j}\right)^2+1}\,\bigg),\ j=1,2.
\eeq 
It holds that $0<\xi_2<\xi_1$ and that the interior of $\Om$ corresponds to the rectangular region $\xi_2<\xi<\xi_1$.

We have the following properties for the first eigenvalue and eigenfunction. 
\begin{lemma}[\cite{HLS:2022:FSD}]\label{lemm:positive}
 The first eigenvalue $\sigma_1^t$ is simple and the first eigenfunction $u_1^t$ does not change the sign in $B_2\setminus \overline{B_1^t}$.
\end{lemma}

Now, we show that $u_1^t$ in bispherical coordinates depends only on $\xi,\theta$ as follows. 
\begin{lemma}\label{lemm:simplify u_1^t}
The first eigenfunction $u_1^t$ depends only on $\xi$ and $\theta$, that is, it holds for some smooth function $A_0(\xi,\theta)$ that
\begin{align}\label{simplify u_1^t}
u_1^t=(\cosh\xi-\cos\theta)^{\frac{n}{2}}\cdot A_0(\xi,\theta).
\end{align}

\end{lemma}
\begin{proof}
We can simplify the parametrization $\mathbf{x}(\xi, \theta, \varphi_1,\dots,\varphi_n):\mathbb{R}\times [0,\pi]\times [0,\pi]^{n-1}\times [0,2\pi)\rightarrow \mathbb{R}^{n+2}$ as $\mathbf{x}(\xi, \theta, \mathbf{x}'):\mathbb{R}\times [0,\pi]\times \mathbb{S}^{n}\rightarrow \mathbb{R}^{n+2}$. Let 
\begin{align}
v_1^t(\mathbf{x}):=\int_{\mathbb{S}^{n}}u_1^t(\xi, \theta, \mathbf{x}')d\mathbf{x}'
\end{align}  
with $d\mathbf{x}'=\sin^{n-1}\varphi_1\cdots\sin\varphi_{n-1}d\varphi_1\cdots d\varphi_{n}$. By Lemma \ref{lemm:positive}, $v_1^t$ is nonzero. Since $v_1^t$ is independent of $\mathbf{x}'$,  we have from \eqnref{laplacian} that
\begin{align}
\Delta v_1^t=\int_{\mathbb{S}^n}\frac{1}{\sqrt{|\mathbf{g}|}}&\left[\frac{\partial}{\partial \xi}\left(\frac{\alpha^{n}\sin^{n}\theta \sin^{n-1}\varphi_1\cdots \sin\varphi_{n-1}}{(\cosh\xi-\cos\theta)^{n}}\frac{\partial u_1^t}{\partial \xi} \right)\right. \notag\\
  &  \left. +\frac{\partial}{\partial \theta}\left(\frac{\alpha^{n}\sin^{n}\theta \sin^{n-1}\varphi_1\cdots \sin\varphi_{n-1}}{(\cosh\xi-\cos\theta)^{n}}\frac{\partial u_1^t}{\partial \theta} \right) \right]d\mathbf{x}'. \label{eq:v_1^t 1}
   \end{align}
On the other hand, by the divergence theorem, $\int_{\mathbb{S}^n}\Delta_{\mathbb{S}^n}u_1^t d\mathbf{x}'=0$,  where $\Delta_{\mathbb{S}^n}$ means the Laplace--Beltrami operator on the unit sphere $\mathbb{S}^n$. This implies that
 \begin{align}
 \int_{\mathbb{S}^n}\sum_{i=1}^{n}\frac{1}{\sqrt{|\mathbf{g}|}}\frac{\partial}{\partial \varphi_i}\left(\sqrt{|\mathbf{g}|}g^{i+2, i+2}\frac{\partial u_1^t}{\partial \varphi_i} \right)d\mathbf{x}'=\frac{1}{g_{33}}\int_{\mathbb{S}^n}\Delta_{\mathbb{S}^n}u_1^td\mathbf{x}'=0. \label{eq:v_1^t 2}
   \end{align}
   By (\ref{eq:v_1^t 1}) and (\ref{eq:v_1^t 2}) together with (\ref{laplacian}), $\Delta v_1^t=\int_{\mathbb{S}^n}\Delta u_1^td\mathbf{x}'=0$. The last equality follows from the fact that $\Delta u_1^t(\mathbf{x})=0$ for $\mathbf{x}\in\Om$. In addition, it is easy to check that $v_1^t$ satisfies the Steklov--Dirichlet boundary conditions in \eqnref{problem}. Thus, by Lemma \ref{lemm:positive}, $v_1^t$ is $u_1^t$ up to a constant. Since $v_1^t$ only depends on $\xi$ and $\theta$, so does $u_1^t$ and the proof is complete.
\end{proof}

From the fact that $u_1^t$ is harmonic in $\Om$, we have the following relation for $A_0(\xi,\theta)$. 

\begin{lemma}\label{lemm:A_0 identity}
Set $s=\cos\theta$, then $A_0(\xi, \theta)$ in \eqnref{simplify u_1^t} satisfies that
    \begin{align} \label{eq:A_0 identity}
        \frac{\partial^2A_0}{\partial \xi^2}+(1-s^2)\frac{\partial^2 A_0}{\partial s^2}-(n+1)s\frac{\partial A_0}{\partial s}-\frac{n^2}{4}A_0=0.
    \end{align}

\end{lemma}
\begin{proof}
For simplicity, we write $\Phi=\sin^{n-1}\varphi_1\cdots \sin \varphi_{n-1}$. By applying \eqnref{laplacian} to \eqnref{simplify u_1^t}, we have
    \begin{align*}
 &\Delta u_1^t \notag \\
        =  &\frac{1}{\sqrt{|\mathbf{g}|}} 
         \left[\frac{\partial}{\partial \xi}\left(\frac{\alpha^{n}\sin^{n}\theta \,\Phi}{(\cosh\xi-\cos\theta)^{n}}\left(\frac{n}{2}(\cosh \xi-\cos \theta)^{\frac{n}{2}-1}(\sinh \xi) A_0+(\cosh\xi-\cos\theta)^{\frac{n}{2}}\frac{\partial A_0}{\partial \xi}\right) \right)\right. \notag\\
  & \qquad \left. +\frac{\partial}{\partial \theta}\left(\frac{\alpha^{n}\sin^{n}\theta\, \Phi}{(\cosh\xi-\cos\theta)^{n}}\left(\frac{n}{2}(\cosh \xi-\cos \theta)^{\frac{n}{2}-1}(\sin \theta) A_0+(\cosh\xi-\cos\theta)^{\frac{n}{2}}\frac{\partial A_0}{\partial \theta}\right) \right) \right]
      \end{align*}
and, thus,
\begin{align*}
&\Delta u_1^t \\
=&\frac{(\cosh \xi-\cos \theta)^{n+2}}{\alpha^{n-2}\sin^n\theta\times \Phi}
\left[\left(-\frac{n}{2}\left(\frac{n}{2}+1\right)\frac{\alpha^n \sin^n\theta\times \Phi\times (\sinh\xi)^2}{(\cosh\xi-\cos\theta)^{\frac{n}{2}+2}}+\frac{n}{2}\frac{\alpha^n \sin^n\theta\times \Phi\times (\cosh\xi)}{(\cosh\xi-\cos\theta)^{\frac{n}{2}+1}}\right)A_0\right. \notag\\
 & \qquad\qquad+\left.\left(-\frac{n}{2}\left(\frac{n}{2}+1\right)\frac{\alpha^n \sin^n\theta\times \Phi\times (\sin\theta)^2}{(\cosh\xi-\cos\theta)^{\frac{n}{2}+2}}+\frac{n(n+1)}{2}\frac{\alpha^n \sin^n\theta\times \Phi\times (\cos\theta)}{(\cosh\xi-\cos\theta)^{\frac{n}{2}+1}}\right)A_0\right. \\
  &\qquad\qquad\left.+\frac{\alpha^n \sin^n\theta\times\Phi}{(\cosh\xi-\cos\theta)^{\frac{n}{2}}}\frac{\partial^2A_0}{\partial \xi^2}+\frac{n\alpha^n \sin^{n-1}\theta\cos\theta\times\Phi}{(\cosh\xi-\cos\theta)^{\frac{n}{2}}}\frac{\partial A_0}{\partial \theta}+\frac{\alpha^n \sin^n\theta\times\Phi}{(\cosh\xi-\cos\theta)^{\frac{n}{2}}}\frac{\partial^2A_0}{\partial \theta^2} \right].
\end{align*}
We derive 
\begin{align*}
0=&\frac{\alpha^2} {(\cosh \xi-\cos \theta)^{\frac{n}{2}+2}}\, \Delta u_1^t \\
    =&
      \left(-\frac{n}{2}\left(\frac{n}{2}+1\right)\frac{\sinh^2\xi+\sin^2\theta}{(\cosh\xi-\cos\theta)^2}+\frac{n}{2}\frac{\cosh\xi}{\cosh\xi-\cos\theta}+\frac{n(n+1)}{2}\frac{\cos\theta}{\cosh\xi-\cos\theta} \right)A_0 \\
  & +\frac{\partial^2A_0}{\partial \xi^2}+n\cot\theta\frac{\partial A_0}{\partial\theta} +\frac{\partial^2 A_0}{\partial \theta^2}
  =    \frac{\partial^2A_0}{\partial \xi^2}+\frac{\partial^2A_0}{\partial \theta^2}+n\cot\theta\frac{\partial A_0}{\partial \theta}-\frac{n^2}{4}A_0,
    \end{align*}
by using the relation that $\sinh^2\xi+\sin^2\theta=\cosh^2\xi-\cos^2\theta$.
Hence, we prove (\ref{eq:A_0 identity}). 
\end{proof}

We express the first eigenfunction using the Gegenbauer polynomials with $\lambda=\frac{n}{2}$, as follows. 
\begin{prop}\label{prop:u_1^t seriesexp}
Set $\mathbf{x}=\mathbf{x}(\xi,\theta,\varphi_1,\dots, \varphi_{n})\in\Om\subset\RR^{n+2}$, $n\geq 1$ as in Section \ref{sec: the first: bispherical coordinates}. The first eigenfunction $u_1^t$ admits the series expression
\begin{align} \label{u_1^t seriesexp}
u_1^t\left(\mathbf{x}\right)=(\cosh\xi-\cos\theta)^{\frac{n}{2}}\,\sum_{m={0}}^\infty & C_m\left( e^{(m+\frac{n}{2})(2\xi_1-\xi)}-e^{(m+\frac{n}{2})\xi} \right)G_{m}^{\left(n/2\right)}(\cos\theta)
\end{align} 
with some constant coefficients $C_m$. 
\end{prop}
\begin{proof}
Since $u_1^t$ is smooth and $\xi>0$ on $\overline{\Om}$, we have $A_0=v\circ \mathbf{x}$ for some smooth function $v$ in Cartesian coordinates. Hence, $A_0(\xi,\theta)$ is smooth on $(\xi,\theta)\in (0,\infty)\times [0,\pi]$, which implies that 
$$\widetilde{A}_0(\xi,s):=A(\xi,\theta)\quad\mbox{with }s=\cos\theta$$ 
belongs to $L^2([-1,1];(1-s^2)^{n/2-1/2}\,ds)$ for each $\xi$. Hence, $\widetilde{A}_0(\xi, s)$ admits the Fourier--Gegenbauer series expansion:
\begin{gather} \notag
\widetilde{A}_0(\xi, s)=\sum_{m=0}^\infty a_m(\xi)\, G_m^{\left(n/2\right)}(s) ,\\\label{Fourier-Legendre}
a_m(\xi)=\frac{1}{\|G_{m}^{(n/2)}\|_{\frac{n}{2}-\frac{1}{2}}}\int_{-1}^1 \widetilde{A}_0(s) G_{m}^{(n/2)}(s)(1-s^2)^{n/2-1/2}ds
\end{gather}
with the norm $\|\cdot\|_{n/2-1/2}$ given in \eqnref{Gegen:norm}. 
On the other hand, in view of \eqnref{gegen:deri}, the first derivatives of the Gegenbauer polynomials are complete and orthogonal in 
$L^2([-1,1];(1-s^2)^{n/2+1/2}\,ds)$.
 Hence, $\frac{\partial A_0}{\partial s}$ admits the series expansion 
 \begin{gather}\notag
 \frac{\partial \widetilde{A}_0}{\partial s}=\sum_{m=1}^{\infty}b_m(\xi)\frac{d}{ds}G_m^{\left(n/2\right)}(s),\\\label{Fourier-Legendre:2}
 b_m(\xi)=\frac{1}{\|\frac{d}{ds}G_{m}^{(n/2)}\|_{\frac{n}{2}+\frac{1}{2}}}\int_{-1}^1 \pd{\widetilde{A}_0}{s}(s)\frac{d G_{m}^{(n/2)}}{ds}(s)(1-s^2)^{n/2+1/2}ds.
 \end{gather}
 
 From \eqnref{expan:gegen} and \eqnref{gegen:deri}, we have
 $$\frac{d}{ds}\left(\frac{d G_{m}^{(n/2)}}{ds}(s)(1-s^2)^{n/2+1/2}\right)
=g_{m}^{(n/2)}\, G_{m}^{(n/2)}(s)(1-s^2)^{n/2-1/2} $$
for some constant $g_{m}^{(n/2)}$ independent of $\widetilde{A}_0$. 
By integration by parts, it follows that
\beq\label{twointe_relations}
\int_{-1}^1 \pd{\widetilde{A}_0}{s}(s)\frac{d G_{m}^{(n/2)}}{ds}(s)(1-s^2)^{n/2+1/2}ds
=m(m+n)\int_{-1}^1 \widetilde{A}_0(s) G_{m}^{(n/2)}(s)(1-s^2)^{n/2-1/2}ds.
\eeq

From \eqnref{Fourier-Legendre}, \eqnref{Fourier-Legendre:2} and \eqnref{twointe_relations}, one can easily check that $a_m(\xi)=b_m(\xi)$.
Thus, we conclude (and similarly for the second derivative) that
\begin{align}\notag
    \frac{\partial \widetilde{A}_0}{\partial s}=\sum_{m=1}^{\infty}a_m(\xi)\frac{dG_m^{\left(n/2\right)}}{ds}(s)
    \quad\mbox{and}  \quad  \frac{\partial^2 \widetilde{A}_0}{\partial s^2}=\sum_{m=2}^{\infty}a_m(\xi)\frac{d^2G_m^{\left(n/2\right)}}{ds^2}(s).
\end{align}

We substitute (\ref{Fourier-Legendre}) into (\ref{eq:A_0 identity}) in Lemma \ref{lemm:A_0 identity}. Since $a_m(\xi)$ admits the integral expression \eqnref{Fourier-Legendre}, for which the integrand is a smooth function, $a_m(\xi)$ is twice differentiable. 
Using \eqnref{Gegen:ode}, we obtain that
\begin{align} \label{eq:a_m}
    \frac{\partial^2a_m}{\partial\xi^2}-\left(\frac{n}{2}+m\right)^2a_m=0,
\end{align}
which implies that $a_m(\xi)=C_{m1} e^{\left(m+\frac{n}{2}\right)\xi}+C_{m2} e^{-\left(m+\frac{n}{2}\right)\xi}$ for some constants $C_{m1}$ and $C_{m2}$ for each $m\ge0$. Since $A_0(\xi_1)=0$ by the Dirichlet boundary condition in \eqnref{problem}, $C_{m1}=-C_m$ and $C_{m2}=C_m\cdot e^{2\left(m+\frac{n}{2}\right)\xi_1}$ for some constant $C_m$. Therefore, we obtain (\ref{u_1^t seriesexp}).
\end{proof}

The eigenfunction $u_1^t$ is harmonic in $\Om=B_2\setminus\overline{B_1^t}$ and satisfies the Robin boundary condition with constant ratio $\sigma_1^t$ on the sphere $\p B_2$. One can extend $u_1^t$ across $\p B_2$ so that
the series expansion in (\ref{u_1^t seriesexp}) converges in the domain given by $\xi_2-\delta\leq \xi\leq \xi_1$ with some  $\delta>0$. From \eqnref{prop:u_1^t seriesexp} and \eqnref{Gegen:norm}, we have
\begin{align*}
&\sum_{m=0}^\infty C_m^2\left( e^{(m+\frac{n}{2})(2\xi_1-\xi)}-e^{(m+\frac{n}{2})\xi} \right)^2 \|G_m^{(n/2)}\|^2
<\infty\quad\mbox{for }\xi_2-\delta\leq\xi\leq \xi_1.
\end{align*}
Thus, if $\xi$ is fixed in $[\xi_2-\delta, \xi_2)$, there exists a positive constant $L=L(\xi)$ such that 
\begin{align}\label{Cm:upperbound}
        \left|C_m\,\big( e^{\left(m+\frac{n}{2}\right)(2\xi_1-\xi)}-e^{\left(m+\frac{n}{2}\right)\xi}\big)\right| \|G_m^{\left(n/2\right)}\|\le L(\xi)\quad\text{for all }m\ge0.
\end{align}
For $\xi$ away from $\xi_1$, it holds that
$$\frac{1}{2}\le 1-e^{-2\left(m+\frac{n}{2}\right)(\xi_1-\xi)} \le1\quad\mbox{for sufficiently large }m.$$
Since $ \|G_m^{\left(n/2\right)}\|$ has a polynomial growth in $m$ (see \eqnref{Gegen:norm}) and $e^{\left(m+\frac{n}{2}\right)(2\xi_1-\xi)}-e^{\left(m+\frac{n}{2}\right)\xi}$ behaves like $e^{\left(m+\frac{n}{2}\right)(2\xi_1-\xi)}$ as $m\rightarrow\infty$ for $\xi$ away from $\xi_1$. From \eqnref{Cm:upperbound} with $\xi=\xi_2-\delta$, we obtain
\beq\label{Cm:uppperbound2}
\left|C_m\big( e^{(m+\frac{n}{2})(2\xi_1-\xi_2)}-e^{(m+\frac{n}{2})\xi_2}\big)\right|
=O\left(e^{-(m+\frac{n}{2})\frac{\delta}{2}}\right)\quad\mbox{as }m\rightarrow\infty.
\eeq


For simplicity, we introduce the following quantities.
\begin{notation}\label{not:coefficients} For each $m\ge 0$, we set
\begin{align}
\notag\ds\widetilde{C}_m&=C_m\big( e^{(m+\frac{n}{2})(2\xi_1-\xi_2)}-e^{(m+\frac{n}{2})\xi_2}\big),\\
\notag\ds c_m&=\left(\tanh\big(\big(m+\frac{n}{2}\big)(\xi_1-\xi_2)\big)\right)^{-1/2}\neq 0. 
\end{align}
\end{notation}

The first eigenfunction on $\p B_2$ can be expressed as 
\begin{align}\label{eq:final series expansion of u_1^t}
 u_1^t(\mathbf{x}(\xi_2,\theta,\dots,\varphi_{n-1}))=(\cosh\xi_2-\cos\theta)^{\frac{n}{2}}\,\sum_{m=0}^\infty \widetilde{C}_m \,G_m^{\left(n/2\right)}(\cos\theta).
\end{align}
Using (\ref{Cm:uppperbound2}) and (\ref{maximum bound}), we can apply the term-by-term differentiation of the of the series in (\ref{u_1^t seriesexp}) with $\xi$ on $\xi=\xi_2$ that converges to $\frac{\partial u_1^t}{\partial \xi}\big|_{\xi=\xi_2}$.
Using \eqnref{normal:u:3d}, \eqnref{eq:final series expansion of u_1^t} and the Steklov boundary condition in (\ref{problem}), we obtain a three-term recurrence relation for $\widetilde{C}_m$ as follows. 
\begin{lemma}\label{lemm: recursive}
    We have 
\beq\label{eq:recursion:Ctilde}
\begin{aligned}
&\big(-2\alpha\sigma_1^t-n\sinh\xi_2+nc_0^2\cosh\xi_2\big)\widetilde{C}_0 - nc_1^2\widetilde{C}_1 = 0,\\
&\big(-2\alpha\sigma_1^t-n\sinh\xi_2+(2m+n)c_m^2\cosh\xi_2\big)\widetilde{C}_m - mc_{m-1}^2\widetilde{C}_{m-1} -  (m+n)c_{m+1}^2\widetilde{C}_{m+1}=0,
\end{aligned}
\eeq
for all $m\ge 1$.
\end{lemma}
\begin{proof}

By (\ref{normal:u:3d}) and (\ref{u_1^t seriesexp}), we have
\begin{align} \label{normal derivative in series}
\notag &\frac{\partial u_1^t}{\partial \mathbf{n}}\Big|_{\partial B_2}
=-\frac{\cosh \xi_2-\cos\theta}{\alpha}\frac{\partial u_1^t}{\partial \xi}\Big|_{\xi=\xi_2}\\
\notag =&-\frac{(\cosh \xi_2-\cos \theta)^{\frac{n}{2}}}{\alpha}\left[\frac{n\sinh \xi_2}{2}\sum_{m=0}^{\infty} C_m \left(e^{\left(m+\frac{n}{2}\right)(2\xi_1-\xi_2)}-e^{\left(m+\frac{n}{2}\right)\xi_2}\right)G_m^{\left(n/2\right)}(\cos \theta)\right.\\
\notag
&\left.+(\cosh \xi_2-\cos \theta)\sum_{m=0}^{\infty}C_m\left(-\left(m+\frac{n}{2}\right)e^{\left(m+\frac{n}{2}\right)(2\xi_1-\xi_2)}-\left(m+\frac{n}{2}\right)e^{\left(m+\frac{n}{2}\right)\xi_2}\right)G_m^{\left(n/2\right)}(\cos \theta)\right] \\
=&-\frac{(\cosh \xi_2-\cos \theta)^{\frac{n}{2}}}{\alpha}\sum_{m=0}^{\infty}\left( \frac{n\sinh \xi_2}{2}\widetilde{C}_m -(\cosh \xi_2-\cos \theta)\left(m+\frac{n}{2}\right)c_m^2\widetilde{C}_m\right) G_m^{\left(n/2\right)}(\cos \theta).
\end{align}
We obtain from \eqnref{Gegen:recur} that $G_0^{\left(n/2\right)}(\cos \theta)=1$, $G_1^{\left(n/2\right)}(\cos \theta)=n\cos \theta$, and 
\beq \label{eq:gegenbauer polynomials}
\begin{aligned}
&\left(m+\frac{n}{2}\right)\cos \theta \cdot G_m^{\left(n/2\right)}(\cos \theta)\\
=&\frac{m+1}{2}G_{m+1}^{\left(n/2\right)}(\cos \theta)+\frac{m+n-1}{2}G_{m-1}^{\left(n/2\right)}(\cos \theta) \quad\text{for all }m\ge 1.
\end{aligned}
\eeq
Note that (\ref{eq:gegenbauer polynomials}) holds for $m\ge0$ by defining $G_{-1}^{\left(n/2\right)}(\cos\theta)=0$. We substitute (\ref{eq:gegenbauer polynomials}) for $m\ge 0$ into (\ref{normal derivative in series}) and obtain
\begin{align*}
& \frac{\partial u_1^t}{\partial \mathbf{n}}\Big|_{\partial B_2}
=-\frac{(\cosh \xi_2-\cos \theta)^{\frac{n}{2}}}{\alpha}\times\\
&\qquad  \sum_{m=0}^\infty \Big(\frac{n\sinh\xi_2}{2}\widetilde{C}_m-\cosh\xi_2\big(m+\frac{n}{2}\big)c_m^2\widetilde{C}_m+\frac{m}{2}c_{m-1}^2\tilde{C}_{m-1}
+\frac{m+n}{2}c_{m+1}^2\widetilde{C}_{m+1}\Big)G_m^{\left(n/2\right)}(\cos \theta).
\end{align*}
 Hence, we prove (\ref{eq:recursion:Ctilde}) by applying the Steklov boundary condition in \eqnref{problem},
 $\frac{\partial u_1^t}{\partial \mathbf{n}}=\sigma_1^tu_1^t$ on $\partial B_2$, and \eqnref{eq:final series expansion of u_1^t}. 
\end{proof}

\section{Asymptotic analysis}

In this section, we consider the case when the distance $\varepsilon:=r_2-r_1-t$ between the two boundary spheres, $\partial B_1^t$ and $\partial B_2$, is sufficiently small and observe asymptotic behavior of $\sigma_1^t$. 
If $\eps$ is sufficiently small, by \eqnref{def:alpha} and \eqnref{def:xi_j}, we have (see, for instance,  \cite{HLS:2022:FSD})
\beq\label{alpha:eps}
\begin{aligned}
\alpha&=r_{*}\sqrt{\varepsilon}+O(\varepsilon\sqrt{\varepsilon})\quad\mbox{with }r_{*}=\sqrt{\frac{2r_1r_2}{r_2-r_1}},\\
\xi_j &= \frac{1}{r_j}\alpha+O(\varepsilon\sqrt{\varepsilon})\quad\text{for $j=1,2$}.
\end{aligned}
\eeq

\subsection{Simplification of the recursive relation for the first eigenfunction}
We additionally introduce the notations:
\begin{notation}\label{not:S_m} We set
\begin{align}\notag
   R_m(\eps)&=\frac{c_m^2\widetilde{C}_m}{c_{m-1}^2\widetilde{C}_{m-1}},\quad m\geq 1,\\\notag
   S_m(\eps)&=-\frac{2\alpha\sigma_1^t+n\sinh\xi_2}{c_m^2(m+n)}+\frac{2m+n}{m+n}\cosh\xi_2,\quad m\geq 0.
   \end{align}
We also define
\begin{align*}
N_1(\eps)&=\inf\left\{m\,:\,R_m(\eps)=0\right\};
\quad N_1(\eps)=\infty\mbox{ if }R_m(\eps)\ne0\mbox{ for all }m\ge 1,\\
N_2(\varepsilon)&=\inf\Big\{m\,:\,S_m(\varepsilon)^2-\frac{4m}{m+n}\le 0\Big\};
\quad N_2(\eps)=\infty\mbox{ if }S_m(\varepsilon)^2-\frac{4m}{m+n}>0\mbox{ for all }m\ge 0.
\end{align*}
\end{notation}

The recursion relation (\ref{eq:recursion:Ctilde}) is equivalent to
\beq\label{eq:recursion:R_m}
\begin{cases}
\ds R_1=nS_0,\\
  \ds  R_{m+1}=-\frac{m}{m+n}\frac{1}{R_m}+S_m\quad&\mbox{for }1\le m<N_1.
\end{cases}
\eeq 
For sufficiently small $\eps$, $S_m(\eps)$ has the strict monotonicity in $m$ and admits a lower bound as in the following lemmas. 
\begin{lemma} \label{lemm:Sn}
There exists $\varepsilon_1>0$ such that, for any $\varepsilon\in(0,\varepsilon_1)$, 
 \begin{align}\notag
 0<S_m(\varepsilon)<S_{m+1}(\varepsilon) \quad\text{for } m\ge 0.     
 \end{align}
\end{lemma}
\begin{proof}
Note that
\beq \label{S_m}
 S_m(\eps)
 =\frac{2m+n}{m+n}\cosh\xi_2-\frac{2\alpha\sigma_1^t+n\sinh\xi_2}{m+n}\, \tanh\left(\left(m+\frac{n}{2}\right)(\xi_1-\xi_2)\right), \quad m\ge 0.
\eeq
Consider the function $h:[0, \infty)\rightarrow \mathbb{R}$ defined by
\begin{align}\notag
h(x)= \frac{2x+n}{x+n}\cosh \xi_2-\frac{2\alpha \sigma_1^t+n\sinh \xi_2}{x+n}\, \tanh\left((\xi_1-\xi_2)\left(x+\frac{n}{2}\right)\right);
\end{align}
then $h(m)=S_m$ for all $m\ge 0$ and
\begin{align*}
h'(x)=&\frac{n}{(x+n)^2}\cosh\xi_2+\left(2\alpha\sigma_1^t+n\sinh\xi_2\right)\\
&\times\left[\frac{1}{(x+n)^2}\tanh\left(\left(\xi_1-\xi_2\right)\left(x+\frac{n}{2}\right)\right)-\frac{\xi_1-\xi_2}{x+n}\operatorname{sech}^2\left(\left(\xi_1-\xi_2\right)\left(x+\frac{n}{2}\right)\right)\right].
\end{align*}
Using the relations $\tanh(y)\ge0$ and $y\operatorname{sech}^2(y)\le 1$ for $y>0$, we obtain
\begin{align}
\notag h'(x)&\ge\frac{1}{(x+n)^2}\left[n-\left(2\alpha\sigma_1^t+n\sinh\xi_2\right)\left(\frac{x+n}{x+n/2}\right)\right]\\
\label{eq:dS:lwbd:exact}&\ge\frac{1}{(x+n)^2}\left[n-2\left(2\alpha\sigma_1^t+\frac{n\alpha}{r_2}\right)+O(\alpha^3)\right].
\end{align}
Because $\sigma_1^t$ is bounded (see, for example, \cite[Theorem 1]{Seo:2021:SOP}) and $\alpha\to0$ as $\varepsilon\to0$, for sufficiently small $\varepsilon$, $h'(x)>0$ holds for all $x\in [0, \infty).$ Thus, $S_m(\varepsilon)<S_{m+1}(\varepsilon)$ for all $m\ge 0$. Furthermore, we can check that $0<S_0(\varepsilon)$ for sufficiently small $\varepsilon$. This finishes the proof.
\end{proof}
\begin{lemma} \label{lemm:S_m_bound}
There exists $\eps_2>0$ such that, for any $\eps\in(0, \eps_2)$, 
\begin{align}
    S_m(\eps)\ge\frac{m+n/2}{m+n}\left[2-\left(2\alpha\sigma_1^{t}+\frac{n\alpha}{r_2}\right)\left(\frac{\alpha}{r_1}-\frac{\alpha}{r_2}\right)+\left(\frac{\alpha}{r_2}\right)^2\right]+O(\alpha^3)\quad \mbox{for }m\ge 0,
\end{align}
 where $O(\alpha^3)$ is uniform in $m$. 
\end{lemma}
\begin{proof}
Since $\tanh(y)\le y$ and $\cosh(y)\ge 1+y^2/2$ for $y>0$, \eqnref{S_m} leads to
\begin{align*}
S_m(\eps)&\geq
\frac{2m+n}{m+n}\left(1+\frac{\xi_2^2}{2}\right)-\frac{2\alpha\sigma_1^t+n\sinh\xi_2}{m+n}\left(\left(m+\frac{n}{2}\right)(\xi_1-\xi_2)\right)\\
&=\frac{m+n/2}{m+n}\Big(2+\xi_2^2-\big(2\alpha\sigma_1^t+n\sinh\xi_2\big)(\xi_1-\xi_2)\Big).
\end{align*}
By applying \eqnref{alpha:eps}, we prove the lemma.
\end{proof}

By \eqnref{eq:recursion:R_m}, we have
$$f_m(R_m)=R_{m+1}\mbox{ with } f_m(x):=- \frac{m}{m+n}\frac{1}{x} + S_m,\quad 1\le m < N_1(\varepsilon),$$
where $f_m$ are functions from $\mathbb{R}\backslash\{0\}$ to $\mathbb{R}$. We denote by $L_m$ and $U_m$ ($L_m<U_m$) the two fixed points of $f_m$, that is, the solutions to $x^2 - S_m x + \frac{m}{m+n} = 0$. 
In other words, for $1\le m <N_2(\eps)$,
\begin{align}\label{eq:L_m and U_m}
  L_m=\frac{1}{2}\bigg(S_m-\sqrt{S_m^2-\frac{4m}{m+n}}\bigg)\quad\mbox{and}\quad U_m=\frac{1}{2}\bigg(S_m+\sqrt{S_m^2-\frac{4m}{m+n}}\bigg).
\end{align}

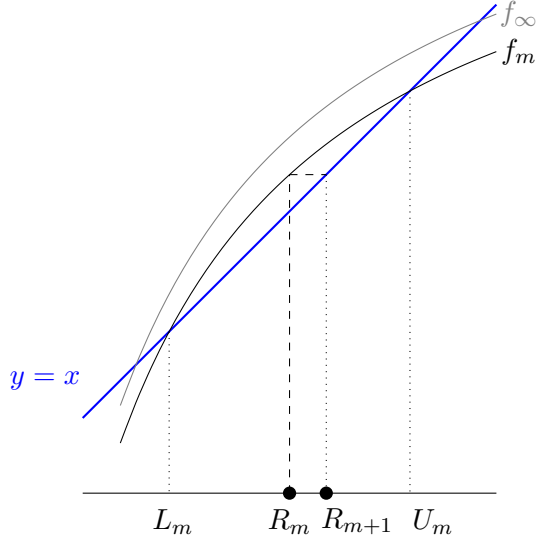
\begin{figure}
    \begin{center}
        \begin{tikzpicture}[scale=5]
\draw (0.5,0.3)--(1.6,0.3);
\draw[blue, thick] (0.5,0.5) --(1.6,1.6);
\node[blue] at (0.4,0.6) {$y=x$};

\draw plot[domain=0.6:1.6,smooth](\x,{-1/\x+2.1});
\draw[gray] plot[domain=0.6:1.6,smooth](\x,{-1/\x+2.1999999976});

\node at (1.6+0.06, -1/1.6+2.1) {$f_{m}$};
\node[gray] at (1.6+0.06, -1/1.6+2.1999999976)  {$f_{\infty}$};

\draw[dotted] (1.37016,1.37016)--(1.37016,0.3)  ;
\node at (1.4316,0.225) {$U_{m}$};

\draw[dotted] (0.729844,0.729844)--(0.729844,0.3);
\node at (0.739844,0.225) {$L_{m}$};

\node at (1.050002, 0.225) {$R_m$};
\fill [black] (1.050002, 0.3) circle (0.5pt);

\draw[dashed] (1.050002, 0.3)--(1.050002, 1.1476208616);

\draw[dashed] (1.050002, 1.1476208616)--(1.1476208616, 1.1476208616);
\draw[dotted] (1.1476208616, 1.1476208616)--(1.1476208616, 0.3);
\node at (1.23, 0.225) {$R_{m+1}$};
\fill [black] (1.1476208616, 0.3) circle (0.5pt);
\end{tikzpicture}
    \end{center}
    
    \caption{Illustration of $R_{m}, R_{m+1}, L_m$ and $U_m$. The two points $(L_m, L_m)$ and $(U_m, U_m)$ are the intersections of $y=f_m(x)$ and $y=x$. Using the recursion relation $R_{m+1}=f_m(R_m)$, $(R_{m+1},0)$ can be obtained from $(R_m, 0)$ via the two graphs.}
    \label{fig:R_{m+1}, L_m, U_m}
\end{figure}

\subsection{Proof of Theorem \ref{thm:main}} \label{sec:main proof}
Fix $\varepsilon$ in  $(0,\eps_1)$, where $\eps_1$ is chosen as in Lemma \ref{lemm:Sn}. Then $S_m(\eps)>0$ for all $m\ge 0$. 
We temporarily assume that $N_1(\varepsilon)=N_2(\varepsilon)=\infty$. Then $L_m$ and $U_m$ are defined and 
\begin{align}\label{eq:0<L_m<R_m}
 0<L_m<U_m\quad\mbox{for each }m\ge 1.
\end{align}
Since $S_{\infty}:=\lim_{m\rightarrow \infty}S_m=2\cosh\xi_2$ by \eqnref{S_m}, we have
\beq\label{eq:LU:infty}
\begin{aligned}
L_\infty&:=\lim_{m\to\infty}L_m = e^{-\xi_2},\quad
U_\infty:=\lim_{m\to\infty}U_m = e^{\xi_2}.
\end{aligned}
\eeq
Note that $(L_{\infty}, L_{\infty})$ and $(U_{\infty}, U_{\infty})$ are the intersections of $y=x$ and $y=f_{\infty}(x):=-\frac{1}{x}+S_{\infty}$.

\begin{lemma}\label{lemm:LnR:asymp:behav}
    Assume $N_1(\varepsilon)=N_2(\varepsilon)=\infty$. Then, we have
    \begin{align}\label{eq:limit of R_n}
        \lim_{m\rightarrow \infty}R_m=L_{\infty}.
    \end{align}
\end{lemma}
\begin{proof} 
Suppose that $R_m$ does not converge to $L_\infty$, that is, there exists a real number $\delta>0$ and a subsequence $R_{m_j}$ of $R_m$ satisfying $|R_{m_j}-L_\infty|>\delta$ for all $j$.
We now fix $\delta_0$ satisfying $0<\delta_0<\min\left(\delta,L_\infty,\frac{U_\infty-L_\infty}{2}\right)$.
By \eqnref{eq:LU:infty}, there exists $N\in\mathbb{N}$ such that
\beq\label{eq:local:uandl}
|U_m-U_\infty|<\delta_0,\quad |L_m-L_\infty|<\delta_0\quad\mbox{for all }m\ge N.
\eeq 
As we assume $N_1(\eps)=\infty$, $R_m$ is nonzero for all $m$. 
We consider the following three cases separately and show that $\lim_{m\rightarrow \infty}R_m$ exists.

\smallskip
\smallskip
\noindent {\bf Case 1} ($L_\infty+\delta_0<R_{k}$ for some $k\ge N$).
From the choice of $\delta_0$ and \eqnref{eq:local:uandl}, it holds that $L_k<R_k$. Hence we have (see Fig. \ref{fig:R_{m+1}, L_m, U_m})
\begin{align}\label{R_k}
R_k<R_{k+1}<U_k\quad\mbox{or}\quad U_k\le R_{k+1}\le R_k.    
\end{align}
In both cases, we have $L_\infty+\delta_0<R_{k+1}$ by the assumption and \eqnref{eq:local:uandl}. By induction, we have $L_\infty+\delta_0<R_m$ for all $m\ge k$, so we have $L_m<R_m$ for all $m\ge k$. In a similar argument as in (\ref{R_k}), we have   
\beq\label{ineq:Rnp1:Case1}
\mathrm{(i)}\ R_m<R_{m+1}<U_m\quad\mbox{or}\quad \mathrm{(ii)}\ U_m\le R_{m+1}\le R_m\quad\mbox{for all }m \ge k.
\eeq
If $\{R_m\}_{m\ge k}$ is a monotone sequence, then $\lim_{m\to\infty} R_m$ exists, because $\{U_m\}_{m\ge k}$ in \eqnref{ineq:Rnp1:Case1} converges to $U_{\infty}$. Otherwise, there exists $k_0\ge k$ such that either (i) holds for $m=k_0$ and (ii) holds for $m=k_0+1$; or (ii) holds for $m=k_0$ and (i) holds for $m=k_0+1$. Hence, it follows that
\beq\label{eq:Rm:notmonotone}
U_{k_0+1} \le R_{k_0+1} < U_{k_0}\quad\mbox{or}\quad U_{k_0} \le R_{k_0+1} < U_{k_0+1}.
\eeq
Then, from \eqnref{eq:local:uandl} and \eqnref{eq:Rm:notmonotone}, we have $|R_{k_0+1}-U_\infty|<\delta_0$.
From \eqnref{eq:local:uandl} and \eqnref{ineq:Rnp1:Case1}, we deduce that $|R_{m+1}-U_\infty|<\delta_0$ for all $m\ge k_0$.
Since we can choose $\delta_0$ to be arbitrarily small, $\lim_{m\to\infty} R_m$ exists and is equal to $U_{\infty}$.

\smallskip
\smallskip
\noindent{\bf Case 2} ($R_{k}<0$ for some $k\ge N$).
We have
$$R_{k+1}=f_k(R_{k})>S_{k}> U_k>U_\infty-\delta_0>L_\infty+\delta_0.$$
This reduces to Case 1.

\smallskip
\smallskip
\noindent{\bf Case 3} ($0<R_{k}<L_\infty-\delta_0$ for some $k\ge N$).
If $R_{m}<0$ for some $m\ge k$, the proof reduces to Case 2.
Thus, we may assume that $R_{m}>0$ for all $m\ge k$.
Because $0<R_{k}<L_\infty-\delta_0<L_k$, we have $$R_{k+1}=f_k(R_k)<R_k<L_\infty-\delta_0.$$
By induction, $\{R_m\}_{m\ge k}$ is a monotone decreasing sequence of positive numbers, so $\lim_{m\to\infty} R_m$ exists.

%

\smallskip

From Case 1 through Case 3, we arrive at $\lim_{m\to\infty}R_m = R_\infty$ for some real number $R_\infty$.
Taking $m\to\infty$ on both sides of the recursion relation $R_{m+1}=f_m(R_m)$, we also have
$$R_\infty = -\frac{1}{R_\infty} + 2\cosh\xi_2,$$
which means that $R_\infty=U_\infty$ by the assumption.

Finally, we prove that $R_{\infty}\ne U_{\infty}$. From (\ref{eq:gegenbauer polynomials}) with $\theta=\frac{\pi}{2}$ and the fact that $G_1^{(n/2)}(0)=0$, we have 
$$G_{2m+2}^{\left(n/2\right)}(0)=-\frac{2m+n}{2m+2}\, G_{2m}^{\left(n/2\right)}(0),\quad G_{2m+1}^{\left(n/2\right)}(0)=0\quad\mbox{for all }m\ge1.$$
 Thus, the ratio test for the convergence of (\ref{eq:final series expansion of u_1^t}) at $\theta=\frac{\pi}{2}$ gives
\begin{align*}
1&\ge \limsup_{m\rightarrow \infty}\left|\frac{\widetilde{C}_{2m+2}\, G_{2m+2}^{\left(n/2\right)}(0)}{\widetilde{C}_{2m}\, G_{2m}^{\left(n/2\right)}(0)}\right|\\
&=\limsup_{m\rightarrow \infty}\left| \frac{c_{2m}^2}{c_{2m+2}^2}\,R_{2m+1}R_{2m+2}\, \frac{2m+n}{2m+2}\right|=|R_\infty|^2.
\end{align*}
Since $U_{\infty}=e^{\xi_2}>1$, we conclude that $R_\infty\ne U_\infty$. It contradicts the assumption.
Therefore, we obtain \eqnref{eq:limit of R_n}.
\end{proof}

\paragraph{Proof of Theorem \ref{thm:main}}
We only need to consider the case for $r_2<\frac{n+1}{n}r_1$.
We assume that \eqnref{ineq:sdeig:lowerbound} does not hold and will derive a contradiction to Lemma \ref{lemm:LnR:asymp:behav}, which proves the theorem. 

By negating \eqnref{ineq:sdeig:lowerbound}, there exists a constant $C$ satisfying $0<C<1$ and a sequence $\{\varepsilon_j\}_{j=1}^\infty$ converging $0$ such that 
$$\sigma_1^{t_j}<C\frac{(n+1)r_1-{n}r_2}{2r_2(r_2-r_1)}\quad\mbox{for all }j,$$
where $t_j:=r_2-r_1-\eps_j$.
From Lemma \ref{lemm:S_m_bound}, we have 
\begin{align}
\notag S_m\notag
&\ge \frac{m+n/2}{m+n}\left[2-2\alpha\sigma_1^{t}\left(\frac{\alpha}{r_1}-\frac{\alpha}{r_2}\right)+ \left(-\frac{n}{r_1r_2}+\frac{n}{r_2^2}+\frac{1}{r_2^2}\right)\alpha^2\right]+O(\alpha^3)\\
\label{eq:S:contra:lwbd}&\ge\frac{m+n/2}{m+n}\left(2+\widetilde{C}\alpha^2\right)+O(\alpha^3)
\end{align}
with $\widetilde{C}:=(-C+1)\frac{{(n+1)}r_1- nr_2}{r_1r_2^2}>0$.
Therefore, we have
\beq\label{eq:D:contra:lwbd}
S_m^2-\frac{4m}{m+n}\ge\frac{n^2}{(m+n)^2}+\left(\frac{2m+n}{m+n}\right)^2\widetilde{C}\alpha^2+O(\alpha^3),
\eeq
which implies $$N_2(\varepsilon_j)=\infty\quad\mbox{for sufficiently large }j.$$
Also, we combine \eqnref{eq:S:contra:lwbd} and \eqnref{eq:D:contra:lwbd} to arrive at
\begin{align}\label{eq:U:lwbd}
U_m &=\frac{1}{2}\left(S_m+\sqrt{S_m^2-\frac{4m}{m+n}}\right)
\ge1+\frac{1}{4}\widetilde{C}\alpha^2+O(\alpha^3),\\    \label{eq:L:lessthan1}
L_m &=\frac{1}{2}\left(S_m-\sqrt{S_m^2-\frac{4m}{m+n}}\right)=\frac{m}{m+n}\frac{1}{U_m}<\frac{1}{U_m}\le 1-\frac{1}{4}\widetilde{C}\alpha^2+O(\alpha^3).
\end{align}

From \eqnref{eq:recursion:R_m} and \eqnref{eq:S:contra:lwbd}, we have
$$R_1(\varepsilon_j)=nS_0(\varepsilon_j)\ge n\left(1+\frac{1}{2}\widetilde{C}\alpha(\varepsilon_j)^2\right)+O(\alpha(\varepsilon_j)^3).$$
Thus, there exists a positive integer $j_1$ satisfying $R_1(\varepsilon_j)>1$ for all $j>j_1$. From \eqnref{eq:L:lessthan1}, we can further assume that $L_\infty(\eps_j)=\lim_{n\rightarrow\infty}L_m(\eps_j)<1$ for all $j>j_1$. 
We prove that $R_m(\varepsilon_j)>1$ for all $m$ and $j>j_1$ as follows. 

Fix $j>j_1$ and suppose that $R_m=R_m(\varepsilon_j)>1$ for some $m>0$. We have the three cases:

\smallskip
\smallskip
\noindent{\bf Case 1} ($R_m>U_m$). It holds that $U_m<R_{m+1}<R_m$; see Case 1 in the proof of Lemma \ref{lemm:LnR:asymp:behav}. Because we have $U_m>1$ from \eqnref{eq:U:lwbd}, we conclude $R_{m+1}>1$.

\smallskip
\smallskip
\noindent{\bf Case 2} ($R_m=U_m$). It holds that $U_m=R_{m+1}=R_m$, so we have $R_{m+1}=R_m>1$.

\smallskip
\smallskip
\noindent{\bf Case 3} ($R_m<U_m$). Since $R_m>1$, we have $L_m<1<R_m<U_m$ by \eqnref{eq:L:lessthan1}. It follows that $R_m<R_{m+1}<U_m$; see Case 1 in the proof of Lemma \ref{lemm:LnR:asymp:behav}. Thus, we have $R_{m+1}>R_m>1$.

\smallskip
By induction, we have 
\beq\label{Rm:relation}
R_m(\varepsilon_j)>1>L_\infty(\varepsilon_j)\quad\mbox{for all }m\geq 1,\ j>j_1,
\eeq
which implies $N_1(\varepsilon_j)=\infty$. Recall that $N_2(\ep_j)=\infty$ for sufficiently large $j$. The relation \eqnref{Rm:relation} contradicts Lemma \ref{lemm:LnR:asymp:behav}. Therefore, we conclude that  \eqnref{ineq:sdeig:lowerbound} holds. 
\qed



\end{document}